\documentclass{amsart}
\usepackage{latexsym,amssymb,amsthm,amsmath}

\usepackage{graphics,amscd}
\usepackage{amssymb}
\usepackage{eucal}
\usepackage{amsmath}

\theoremstyle{plain}
\newtheorem{theorem}{Theorem}[section]
\newtheorem*{Theorem B}{Theorem B}
\newtheorem*{Theorem A}{Theorem A}

\newtheorem{proposition}{Proposition}[section]
\newtheorem{corollary}{Corollary}[section]

\newtheorem{example}{Example}[section]
\numberwithin{equation}{section}

\theoremstyle{remark}
\newtheorem{remark}{Remark}[section]
 \numberwithin{equation}{section}

\def\<{\left < }
\def\>{\right >}
\def\({\left ( }
\def\){\right )}

\def\e{\eqref}

\title[CR-warped submanifolds in Kaehler manifolds]
{Geometry of warped product and CR-warped product submanifolds in Kaehler manifolds: modified version}

\author{ Bang-Yen Chen}


\begin{document}

\begin{abstract} 
The warped product $N_1\times_f N_2$ of two Riemannian manifolds $(N_1,g_1)$ and $(N_2,g_2)$ is the product manifold $N_1\times N_2$ equipped with the warped product metric $g=g_1+f^2 g_2$, where $f$ is a positive function on $N_1$. 
Warped products play very important roles in differential geometry as well as in physics. 
A submanifold $M$ of a Kaehler manifold $\tilde M$ is called a $CR$-warped product if it is a warped product $M_T\times_f N_\perp$ of a complex submanifold $M_T$ and a totally real submanifold $M_\perp$ of $\tilde M$.

 In this article we survey recent results on  warped product and $CR$-warped product submanifolds in Kaehler manifolds. 
Several closely related results will also be presented.

\end{abstract}

\maketitle

\setcounter{equation}{0}

\section{{Introduction}}

\def\<{\left < }
\def\>{\right >}

Let $B$ and $F$ be two Riemannian manifolds with Riemannian metrics $g_B$ and $g_F$, respectively, and let
$f$ be a positive function on $B$. Consider the product manifold $B\times F$ with
its projection $\pi:B\times F\to B$ and $\eta:B\times F\to F$. The {\it warped product}
$M=B\times_f F$ is the manifold $B\times F$ equipped with the warped product Riemannian metric given by
\begin{equation}\label{E:warped} g=g_B+f^2 g_F\end{equation}
We call the function $f$  the {\it warping function\/} of the warped product \cite{bishop}. The notion of warped products plays important roles in differential geometry as well as in physics, especially in the theory of general relativity (cf. \cite{c11,oneill}). 

A submanifold $M$ of a Kaehler manifold $(\tilde M, \tilde g,J)$ is called a {\it $CR$-submanifold} if there exist a holomorphic distribution $\mathcal D$ and a totally real distribution $\mathcal D^\perp$ on $M$ such that $TM=\mathcal D\oplus \mathcal D^\perp$, where $TM$ denotes the tangent bundle of $M$. The notion of $CR$-submanifolds was introduced by A. Bejancu  (cf. \cite{bejancu}). 

On the other hand, the author proved in \cite{c7} that there do not exist warped product submanifolds of the form $M_\perp\times_f N_T$ in any Kaehler manifold $\tilde M$ such that $N_T$ is a holomorphic submanifold and $N_\perp$ is a totally real submanifold of $\tilde M$. Moreover, the author introduced the notion of {\it $CR$-warped products}  in \cite{c7} as follows.
A submanifold $M$ of a Kaehler manifold $\tilde M$ is called a $CR$-warped product if it is a warped product $M_T\times_f N_\perp$ of a complex submanifold $M_T$ and a totally real submanifold $M_\perp$ of $\tilde M$.

A famous embedding theorem of J. F. Nash \cite{nash} states that every Riemannian manifold can be isometrically imbedded in a Euclidean space with sufficiently high codimension.  In particular, the Nash theorem implies that {\it every warped product manifold $N_1\times_f N_2$ can be isometrically embedded as a Riemannian submanifold in a Euclidean space.}  

In view of Nash's theorem, the author asked at the beginning of this century the following two fundamental questions (see \cite{c02,c02-2,c11}).
\vskip.05in

{\bf Fundamental Question A.} {\it What can we conclude from an arbitrary isometric
immersion of a warped product manifold into a Euclidean space or more generally, into an arbitrary Riemannian manifold\,?} 
\vskip.03in

{\bf Fundamental Question B.} {\it What can we conclude from an arbitrary $CR$-warped product manifold into an arbitrary complex-space-form or more generally, into an arbitrary
Kaehler manifold\,?} 
 \vskip.03in
 
 The study of these two questions was  initiated  by the author  in a series of his articles [11, 13--15, 17--25, 29, 34].  Since then the study of  warped product submanifolds has become an active research subject in differential geometry of submanifolds.
 
The  purpose of this article is to survey  recent results on  warped product and $CR$-warped product submanifolds in Kaehler manifolds. Several closely related results will also be presented.

\section{Preliminaries} In this section we provide some basic notations, formulas, definitions, and results.

For the submanifold $M$ we denote by $\nabla$ and ${\tilde \nabla}$ the Levi-Civita connections of $M$ and $\tilde M^m$, respectively. The Gauss and Weingarten
formulas are given respectively by (see, for instance, \cite{cbook,c11,c15})
\begin{align} &{\tilde \nabla}_{X}Y=\nabla_{X} Y +\sigma(X,Y),\\ &{\tilde\nabla}_{X}\xi =
A_{\xi}X+D_{X}\xi \end{align} 
for any  vector fields $X,Y$ tangent to $M$ and  vector field $\xi$ normal to $M$, where $\sigma$ denotes the second fundamental form, $D$ the normal connection, and $A$ the shape operator
of the submanifold. 

 Let $M$ be an $n$-dimensional submanifold of a
 Riemannian $m$-manifold $\tilde M^m$.   We choose a local field of orthonormal
frame
$e_1,\ldots,e_n,e_{n+1},\ldots,e_m$ in $\tilde M^m$ such that, restricted to  $M$, the vectors $e_1,\ldots,e_n$ are tangent to $M$ and hence $e_{n+1},\ldots,e_m$ are normal to $M$. 
Let $\{\sigma^r_{ij}\}$, $i,j=1,\ldots,n;\,r=n+1,\ldots,m$,  denote the coefficients of the second fundamental form $h$ with respect to $e_1,\ldots,e_n,e_{n+1},\ldots,e_m$. Then we have $$\sigma^r_{ij}=\<\sigma(e_i,e_j),e_r\>=\<A_{e_r}e_i,e_j\>,$$ where $\<\;\,,\;\>$ denotes the inner product.

The mean curvature vector $\overrightarrow{H}$ is defined by
\begin{align}\overrightarrow{H} = {1\over n}\,\hbox{\rm trace}\,\sigma = {1\over n}\sum_{i=1}^{n} \sigma(e_{i},e_{i}), \end{align}
where $\{e_{1},\ldots,e_{n}\}$ is a local  orthonormal frame of the tangent bundle $TM$ of $M$. The squared mean curvature is then given by $$H^2=\left<\right.\hskip-.02in \overrightarrow{H},\overrightarrow{H}\hskip-.02in\left.\right>.$$ A submanifold $M$  is called {\it minimal}  in $\tilde M^m$ if  its mean curvature vector  vanishes identically. 

Let $R$ and $\tilde R$ denote the Riemann curvature tensors of $M$ and $\tilde M^m$, respectively. The {\it equation of Gauss\/} is given  by \begin{equation}\begin{aligned}\label{2.4}  &R(X,Y;Z,W)=\tilde R(X,Y;Z,W)+\<\sigma(X,W),\sigma(Y,Z)\>\\&\hskip1.2in  -\<\sigma(X,Z),\sigma(Y,W)\>\end{aligned}\end{equation} 
for vectors $X,Y,Z,W$ tangent to $M$. For a submanifold of a Riemannian manifold of constant curvature $c$, we have 
\begin{equation}\begin{aligned}\label{2.5} R(X&,Y;Z,W)=c\{\left<X,W\right>\left<Y,Z\right>-
\left<X,Z\right>\left<Y,W\right>\}\\ & +\left<\sigma(X,W),\sigma(Y,Z)\right>
-\left<\sigma(X,Z),\sigma(Y,W)\right>.\end{aligned}\end{equation}

Let $M$ be a Riemannian $p$-manifold and  $e_1,\ldots,e_p$ be an orthonormal frame fields on $M$. For differentiable function $\varphi$ on $M$, the Laplacian $\Delta\varphi$ of $\varphi$ is defined  by
\begin{equation} \Delta\varphi=\sum_{j=1}^p \{(\nabla_{e_j}e_j)\varphi-e_je_j\varphi\}.\end{equation}
For any orthonormal basis $e_1,\ldots,e_n$ of the tangent space $T_pM$ at a point $p\in M$, the scalar curvature $\tau$ of $M$ at $p$ is defined to be (cf. [9, 10, 26])
\begin{align}\tau(p)=\sum_{i<j} K(e_i\wedge e_j),\end{align} where $K(e_i\wedge e_j)$ denotes the sectional curvature of the plane section spanned by $e_i$ and $e_j$.

\section{Warped products in space forms}

Let $M_1,\ldots,M_k$ be  $k$ Riemannian manifolds and let  $$f:M_1\times\cdots\times M_k\to \mathbb E^N$$ be an isometric immersion of the Riemannian product $M_1\times\cdots\times M_k$ into the Euclidean $N$-space $\mathbb E^N$. J. D. Moore
\cite{moore}  proved that if the second fundamental form $\sigma$ of $f$ has the property that $\sigma(X,Y)=0$ for $X$  tangent to $M_i$ and $Y$ tangent to $M_j$, $i\ne j$, then $f$ is a product immersion, that is, there exist isometric immersions $f_i:M_i\to  E^{m_i},\, 1\leq i\leq k$, such that
\begin{equation}f(x_1,\ldots,x_k)=(f(x_1),\ldots,f(x_k))\end{equation}
when $x_i\in M_i$ for $1\leq i\leq k$.

 Let $\phi:N_1\times_f N_2\to R^m(c)$ be an isometric immersion of a warped product $N_1\times_f N_2$ into a Riemannian manifold with  constant sectional curvature $c$. Denote by $\sigma$ the second fundamental form of $\phi$. 
  The immersion $\phi:N_1\times_f N_2\to R^m(c)$ is called {\it mixed totally geodesic\/} if $\sigma(X,Z)=0$ for any $X$ in $\mathcal D_1$ and $Z$ in  $\mathcal D_2$.

The next theorem provides a solution to Fundamental Question A.

\begin{theorem}\label{T:3.1} \cite{c02} For any  isometric immersion \ $\phi:N_1\times_f N_2\to R^m(c)$ of a warped product $\,N_1\times_f N_2$  into a  Riemannian manifold of
constant curvature $c$, we have
\begin{equation}\label{E:wpfirst} \frac{\Delta f} {f}\leq \frac{n^2}{4n_2}H^2+ n_1c,\end{equation}
 where $n_i=\dim N_i$, $n=n_1+n_2$,  $H^2$ is the squared mean curvature of $\phi$, and $\Delta f$ is the Laplacian of $f$ on  $N_1$. 

 The equality sign of \eqref{E:wpfirst}  holds identically if and only if $\iota :N_1\times_f N_2\to R^m(c)$ is a mixed totally geodesic immersion satisfying ${\rm trace}\, h_1={\rm trace}\,h_2$, where ${\rm trace}\, h_1$ and ${\rm trace }\,h_2$ denote the trace of $\sigma$ restricted to $N_1$ and $N_2$, respectively.\end{theorem}  
   
By  making a minor modification of the proof of Theorem \ref{T:3.1} in \cite{c02}, using the method of \cite{c05}, we also have the following  general solution from \cite{CWei} to the Fundamental Question A.

\begin{theorem}\label{T:3.2} If $\tilde M^m_c$ is a Riemannian manifold with sectional curvatures bounded from above by a constant $c$, then for any isometric immersion $\phi : N_1 \times _f N_2 \to \tilde M^m_c$   from a warped product $N_1 \times _f N_2$ into  $\tilde M^m_c$ the warping function $f$ satisfies
\begin{align}\label{3.1} \frac{\Delta f} {f}\leq \frac{n^2}{4n_2}H^2+ n_1c,\end{align}
where $n_1= \dim N_1$ and $n_2=\dim N_2$.  \end{theorem}

An immediate consequence of Theorem \ref{T:3.2} is the following.

\begin{corollary} \label{C:3.1} There do not exist  minimal  immersions of a Riemannian product $N_1\times N_2$ of two Riemannian manifolds into a negatively curved Riemannian manifold $\tilde M$.
\end{corollary}
  
For arbitrary  warped products submanifolds in complex hyperbolic spaces, we have the following general results from \cite{c02-4}.

\begin{theorem}\label{T:3.3} Let $\phi:N_1\times_f N_2\to CH^m(4c)$ be an arbitrary isometric immersion of a warped product $N_1\times_f N_2$ into the complex hyperbolic $m$-space $CH^m(4c)$ of constant holomorphic sectional curvature $4c$. Then we have
\begin{align} \label{E:1.1} {\Delta f\over {f}}\leq {{(n_1+n_2)^2}\over{4n_2}}H^2+ n_1c. \end{align} 

The equality sign of \eqref{E:1.1} holds if and only if the following three conditions hold.
\begin{enumerate}

\item  $\phi$ is mixed totally geodesic, 

\item ${\rm trace}\,h_1= {\rm trace}\,h_2$, and

\item $J{\mathcal D_1}\perp {\mathcal D_2}$, where $J$ is the almost complex structure of $CH^m$.
\end{enumerate}\end{theorem}

Some interesting immediate consequences of Theorem \ref{T:3.3} are the following  non-existence results  \cite{c02-4}.

\begin{corollary}\label{C:3.2} Let $N_1\times_f N_2$ be a warped product whose warping function $f$ is  harmonic. Then  $N_1\times_f N_2$ does not admit an isometric minimal immersion into any complex hyperbolic space.
 \end{corollary}

\begin{corollary}\label{C:3.3} If $f$ is an eigenfunction of Laplacian on $N_1$ with eigenvalue $\lambda>0$, then  $N_1\times_f N_2$ does not admits an isometric minimal immersion into any complex hyperbolic space.
\end{corollary}

\begin{corollary}\label{C:3.4}If $N_1$ is compact, then every warped product $N_1\times_f N_2$ does not admit an isometric minimal immersion into any complex hyperbolic space.
\end{corollary}

For arbitrary  warped products submanifolds in the complex projective $m$-space $CP^{m}(4c)$ with constant holomorphic sectional curvature $4c$, we have the following results from \cite{c03-2}.

\begin{theorem}\label{T:3.4} Let $\phi:N_1\times_f N_2\to CP^m(4c)$ be an arbitrary isometric immersion of a warped product into the complex projective $m$-space $CP^m(4c)$ of constant holomorphic sectional curvature $4c$. Then we have 
\begin{align} \label{F:1.2}{\Delta f\over{f}}\leq {{(n_1+n_2)^2}\over{4n_2}}H^2+ (3+n_1)c.\end{align} 

The equality sign of \eqref{F:1.2} holds identically if and only if  we have 

\begin{enumerate}

\item  $n_1=n_2=1$,  

\item  $f$ is an eigenfunction of the Laplacian of $N_1$ with eigenvalue $4c$, and  

\item $\phi$ is totally geodesic and holomorphic.
\end{enumerate}
\end{theorem}
 
An immediate  application of Theorem \ref{T:3.4} is the following non-immersion result.

\begin{corollary} \label{C:3.4}  If $f$ is a positive function on a Riemannian $n_1$-manifold $N_1$ such that $(\Delta f)/f>3+n_1$ at some point  $p\in N_1$, then, for any Riemannian manifold $N_2$, the warped product  $N_1\times_f N_2$ does not admit any isometric minimal immersion into  $CP^m(4)$ for any $m$.
\end{corollary}

Theorem \ref{T:3.4} can be sharpen as the following theorem for totally real minimal immersions.

\begin{theorem} \label{T:3.5}  If $f$ is a positive function on a Riemannian $n_1$-manifold $N_1$ such that $(\Delta f)/f>n_1$ at some point  $p\in N_1$, then, for any Riemannian manifold $N_2$, the warped product  $N_1\times_f N_2$ does not admit any isometric totally real minimal immersion into  $CP^m(4)$ for any $m$.\end{theorem}

The following examples illustrate  that Theorems \ref{T:3.3}, \ref{T:3.4} and \ref{T:3.5} are sharp.

\begin{example} {\rm Let $I=(-\frac{\pi}{4},\frac{\pi}{4}),\, N_2=S^1(1)$ and $f=\frac{1}{2}\cos 2 s$. Then the warped product
$$N_1\times_f N_2=: I\times_{(\cos 2 s)/2}S^1 (1)$$ has constant sectional curvature 4. Clearly, we have $(\Delta f)/f=4$. If we define the complex structure $J$ on the warped product by
$J\left({\partial\over {\partial s}}\right)=2(\sec 2 s) {\partial\over {\partial t}},$
 then $\,(I\times_{(\cos 2 s)/2}S^1(1),g,J)\,$ is  holomorphically isometric to a dense open subset of $CP^1(4)$.

Let $\phi:CP^1(4)\to CP^m(4)$ be a standard totally geodesic embedding of $CP^1(4)$ into $CP^m(4)$. Then the restriction of $\phi$ to $I\times_{(\cos 2 s)/2}S^1(1)$ gives rise to a minimal isometric immersion of $I\times_{(\cos 2 s)/2}
S^1$ into $CP^m(4)$ which  satisfies the equality case of inequality  \eqref{F:1.2} on $I\times_{(\cos 2 s)/2} S^1(1)$  identically. }
\end{example}

\begin{example} {\rm Consider the same warped product $N_1\times_f N_2=I\times_{(\cos 2 s)/2}S^1(1)$ as given in Example 5.1. Let $\phi:CP^1(4)\to CP^m(4)$ be the  totally geodesic holomorphic embedding of $CP^1(4)$ into $CP^m(4)$. Then the restriction of $\phi$ to $N_1\times_f N_2$ is an isometric minimal immersion of $N_1\times_f N_2$ into $CP^m(4)$ which satisfies $(\Delta f)/f=3+n_1$ identically. 
This example shows that the assumption ``$(\Delta f)/f>3+n_1$ at some point in $N_1$'' given in Theorem \ref{T:3.4} is best possible.}\end{example}

\begin{example} {\rm Let  $g_1$ be the standard metric on $S^{n-1}(1)$. Denote by $N_1\times_f N_2$ the warped product given by $N_1=(-\pi/2,\pi/2),\,N_2=S^{n-1}(1)$ and $f=\cos s$. Then the warping function of this warped product satisfies $\Delta f/{f}=n_1$
identically. Moreover, it is easy to verify that this warped product is isometric to a dense open subset of $S^n$.
Let \begin{align}\phi\,:\, & S^n(1)\xrightarrow[\text{2:1}]{\text{projection}}
RP^{n}(4) \xrightarrow[\text{totally real}]{\text{totally geodesic}}  CP^n(4)\notag\end{align}
be a standard totally geodesic Lagrangian immersion of $S^n(1)$ into $CP^n(4)$. Then the restriction of $\phi$ to $N_1\times_f N_2$ is a totally real minimal
immersion.
This example illustrates that the assumption ``$(\Delta f)/f>n_1$ at some point in $N_1$'' given in Theorem \ref{T:3.5} is also sharp.}\end{example}

\section{Segre imbedding and its converse}

For simplicity, we denote $S^{n}(1),\, RP^{n}(1)$, $CP^{n}(4)$ and $ CH^{n}(-4)$  by $S^{n},\, RP^{n}$, $CP^{n}$ and $CH^{n}$, respectively.

Let $(z_0^i,\ldots,z_{\alpha_i}^i),$ $1\leq i\leq s,$ be homogeneous coordinates of $CP^{\alpha_i}$. Define a map:
\begin{equation}S_{\alpha_1\cdots \alpha_s}:CP^{\alpha_1}\times\cdots\times CP^{\alpha_s}\to CP^N,\quad N=\prod_{i=1}^s (\alpha_i+1)-1,\end{equation} 
which maps a point $((z_0^1,\ldots,z_{\alpha_1}^1),\ldots,(z_0^s,\ldots,z_{\alpha_s}^s))$ in $CP^{\alpha_1}\times\cdots\times CP^{\alpha_s}$ to the point $(z^1_{i_1}\cdots z^s_{i_j})_{1\leq i_1\leq \alpha_1,\ldots,1\leq i_s\leq \alpha_s}$ in $CP^N$. The map $S_{\alpha_1\cdots \alpha_s}$ is a Kaehler embedding which is known as the {\it Segre embedding}. The Segre embedding was constructed by C. Segre in 1891.  
 
The following results from \cite{c2,CK} established in 1981 can be regarded as the ``converse'' to Segre embedding constructed in 1891.  

\begin{theorem}\label{T:4.1} Let $M_1^{\alpha_1},\ldots,M_s^{\alpha_s}$ be Kaehler manifolds of
dimensions $\alpha_1,\ldots,$ $\alpha_s$, respectively. Then every holomorphically isometric immersion
$$f:M_1^{\alpha_1}\times\cdots\times M_s^{\alpha_s}\to CP^N,\quad N=\prod_{i=1}^s (\alpha_i+1)-1,$$ 
of $M_1^{\alpha_1}\times\cdots\times M_s^{\alpha_s}$ into $CP^N$   is locally  the Segre embedding, i.e., $M_1^{\alpha_1},\ldots,M_s^{\alpha_s}$ are open portions of $CP^{\alpha_1},\ldots, CP^{\alpha_s}$, respectively. Moreover,  $f$ is congruent to
the Segre embedding.\end{theorem}

Let $\bar \nabla^k\sigma$, $k=0,1,2,\cdots$, denote the $k$-th covariant derivative of the second fundamental form. Denoted by $||\bar\nabla^k\sigma||^2$ the squared norm of $\bar \nabla^k\sigma$. 

The following result was proved in  \cite{CK}.
 
\begin{theorem}\label{T:4.2} Let
$M_1^{\alpha_1}\times\cdots\times M_s^{\alpha_s}$ be a product Kaehler submanifold of  $CP^N$. Then 
\begin{equation}\label{4.2} ||\bar \nabla^{k-2}\sigma||^2\geq k!\,2^k\sum_{i_1<\cdots<i_k}\alpha_1\cdots\alpha_k,\end{equation} for $k=2,3,\cdots.$

The equality sign of \eqref{4.2}  holds for some $k$ if and only if   $M_1^{\alpha_1},\ldots,M_s^{\alpha_s}$ are open parts of $CP^{\alpha_1},\ldots, CP^{\alpha_s}$, respectively, and the immersion is congruent to the Segre embedding.\end{theorem}

If $k=2$, Theorem \ref{T:4.2} reduces to the following result of \cite{c2}.

\begin{theorem}\label{T:4.3} Let $M_1^{h}\times M_2^{p}$ be a product Kaehler submanifold of  $CP^N$. Then we have
\begin{equation}\label{4.3} ||\sigma||^2\geq 8hp.\end{equation}

The equality sign of inequality \eqref{4.3}  holds if and only if  $M_1^h$ and $M_2^p$ are open portions of $CP^{h}$ and $CP^{p}$, respectively, and mreover the immersion is congruent to the Segre embedding $S_{h,p}$.
\end{theorem}

We may extend Theorem \ref{T:4.3} to the following for warped products.

\begin{theorem}\label{T:s4} Let $(M_1^{h},g_{1})$ and $(M_2^{p},g_{2})$ be two Kaehler manifolds of complex dimension $h$ and $p$ respectively and let $f$ be a positive function on $M_{1}^{h}$. If $\phi:M_1^{h}\times_{f} M_2^{p}\to CP^{N}$ is a holomorphically isometric immersion of the warped product manifold $M_1^{h}\times_{f} M_2^{p}$ into   $CP^N$. Then $f$ is a constant, say $c$. Moreover, we have
\begin{equation}\label{4.4} ||\sigma||^2\geq 8hp.\end{equation}

The equality sign of \eqref{4.4}  holds if and only if  $(M_1^h,g_{1})$ and $(M_2^p,cg_{2})$ are open portions of $CP^{h}$ and $CP^{p}$, respectively, and moreover the immersion $\phi$ is congruent to the Segre embedding.
\end{theorem}
\begin{proof} Under the hypothesis, the warped product manifold $M_1^{h}\times_{f} M_2^{p}$ must be  a Kaehler manifold. Therefore, the warping function $f$ must be a positive constant. Consequently, the theorem follows from Theorem \ref{T:4.3}. 
\end{proof}

\vskip.1in

\section{$CR$-products in Kaehler manifolds}
 
 A submanifold $N$ in a Kaehler manifold $\tilde M$ is called a {\it totally real submanifold} if the almost complex structure $J$ of $\tilde M$ carries each tangent space $T_xN$ of $N$ into its corresponding normal space $T_x^\perp N$  \cite{c6,c8.1,CO}. The submanifold $N$ is called a
holomorphic submanifold (or Kaehler submanifold) if $J$ carries each $T_xN$ into itself. The submanifold $N$ is called {\it slant} \cite{c1990}  if for any nonzero vector $X$ tangent to $N$
the angle $\theta(X)$ between $JX$ and $T_pN$  does not depend on
the choice of the point $p\in N$ and of the choice of the vector $X\in T_pN$.  On the other hand, it depends only on the point $p$, then the submanifold $N$ is called {\it pointwise slant} \cite{CG}.

Let $M$ be a submanifold of a Kaehler manifold $\tilde M$.  For each point $p\in M$, put
$${\mathcal H}_p=T_p M\cap J(T_pM),$$ i.e., $\mathcal H_p$ is the maximal holomorphic subspace of the tangent space $T_pM$. If the dimension of $\mathcal H_p$ remains the same for each $p\in M$, then $M$ is called a {\it generic submanifold}  \cite{c1981}.

A $CR$-submanifold  of a Kaehler manifold $\tilde M$ is called a
{\it $CR$-product} \cite{c2,c3} if it is a Riemannian product  $N_T\times N_\perp$ of a Kaehler submanifold $N_T$ and a totally real submanifold $N_\perp$. A $CR$-submanifold  is called {\it mixed totally geodesic} if the second fundamental form of the $CR$-submanifold satisfying $$\sigma(X,Z)=0$$ for any $X\in \mathcal D$ and $Z\in \mathcal D^{\perp}$.

For $CR$-products in complex space forms, the following result  are known. 

\begin{theorem}\label{T:5.1} \cite{c2} We have
\begin{enumerate}
\item[(i)] A $CR$-submanifold in the complex Euclidean $m$-space $\mathbb C^m$  is a $CR$-product if and only if it is a direct sum of a Kaehler submanifold and a totally real submanifold of  linear complex subspaces. 

\item[(ii)] There do not exist $CR$-products in complex hyperbolic spaces other than Kaehler submanifolds and totally real submanifolds.
\end{enumerate}\end{theorem}

 $CR$-products $N_T\times N_\perp$ in $CP^{h+p+hp}$ are obtained from the Segre embedding $S_{h,p}$; namely, we have the following results.

\begin{theorem}\label{T:5.2}  \cite{c2} Let $N_T^{h}\times N_\perp^p$ be the $CR$-product in $CP^m$ with constant holomorphic sectional curvature $4$. Then 
\begin{equation}\label{5.1} m\geq h+p+ hp. \end{equation}
 
The equality sign of \eqref{5.1} holds if and only if 

\begin{enumerate}
\item[(a)] $N_T^h$ is a totally geodesic Kaehler submanifold, 

\item[(b)] $N_\perp^p$ is a totally real submanifold, and 

\item[(c)] the immersion is given by
\begin{equation}\notag N_T^h\times N_\perp^p \xrightarrow{\text{{\rm}}} CP^h\times CP^p\xrightarrow [\text{{\rm Segre imbedding}}]{\text{$S_{hp}$}} CP^{h+p+hp}.\end{equation} 
\end{enumerate}\end{theorem}  

\begin{theorem} \label{T:5.3} \cite{c2}  Let $N_T^{h}\times N_\perp^p$ be a $CR$-product in $CP^m$. Then the squared norm of the second fundamental form  satisfies
\begin{equation}\label{5.2} ||\sigma||^2\geq 4 hp.\end{equation}
 
The equality sign of \eqref{5.2} holds if and only if 

\begin{enumerate}
\item[(a)] $N_T^h$ is a totally geodesic Kaehler submanifold, 

\item[(b)] $N_\perp^p$ is a totally geodesic totally real submanifold, and 

\item[(c)] the immersion is given by
\begin{equation}\notag N_T^h\times N_\perp^p \xrightarrow{\text{{\rm totally geodesic}}} CP^h\times CP^p\xrightarrow [\text{{\rm Segre imbedding}}]{\text{$S_{hp}$}} CP^{h+p+hp}\subset CP^m.
\end{equation} 
\end{enumerate}\end{theorem}

\section{Warped Product $CR$-submanifolds}
 
In this section we present known results on warped product $CR$-submanifold in Kaehler manifolds.
First, we mention the following result. 

\begin{theorem}\label{T:6.1} \cite{c7} If  $\,N_\perp\times_f N_T$ is a warped product $CR$-submanifold of a Kaehler manifold $\tilde M$ such that $N_\perp$ is a totally real and $N_T$  a Kaehler submanifold of $\tilde M$, then it is a $CR$-product.\end{theorem}

Theorem \ref{T:6.1} shows that there does not exist warped product $CR$-submanifolds of the form $N_\perp\times_f N_T$  other than $CR$-products. So, we only  need to consider warped product $CR$-submanifolds of the form: $N_T\times_f N_\perp$, {\it by reversing the two factors $N_T$ and $N_\perp$} of the warped product. The author simply calls such  $CR$-submanifolds $CR$-{\it warped products} in \cite{c7}.
  
$CR$-{\it warped products} are simply characterized  as follows.

\begin{proposition}\label{P:8.1}   \cite{c7} A proper $CR$-submanifold $M$ of a Kaehler manifold $\tilde M$ is locally a $CR$-warped product if and only if the shape operator $A$ satisfies \begin{equation}A_{JZ}X=((JX)\mu)Z,\quad X\in \mathcal D,\quad Z\in \mathcal D^\perp,\end{equation}
for some  function $\mu$ on $M$ satisfying $W\mu=0,\,\forall W\in \mathcal D^\perp$. \end{proposition} 

A fundamental result on $CR$-warped products in arbitrary Kaehler manifolds is  the following  theorem.

\begin{theorem} \label{T:6.2}   \cite{c7,c18} Let $N_T\times_f N_\perp$ be a  $CR$-warped product submanifold in an arbitrary Kaehler manifold $\tilde M$. Then the second fundamental form $\sigma$ satisfies
\begin{equation}\label{6.2}||\sigma||^2\geq 2p\,||\nabla(\ln f)||^2,\end{equation} where $\nabla (\ln f)$ is the gradient of $\,\ln f$ on $N_T$ and $p=\dim N_\perp$.

If the equality sign of \eqref{6.2}  holds identically, then $N_T$ is a totally geodesic Kaehler submanifold and $N_\perp$ is a totally
umbilical totally real submanifold of $\tilde M$. Moreover, $N_T\times_f N_\perp$ is  minimal in $\tilde M$.

 When $M$ is anti-holomorphic, i.e., when $J\mathcal D^\perp_x=T^\perp_xN$, and $p>1$. The equality sign of \eqref{6.2}  holds identically if and only if   $N_\perp$ is a totally umbilical submanifold of $\tilde M$.

 Let $M$ be anti-holomorphic with $p=1$. The equality sign of  \eqref{6.2} holds identically if
the characteristic vector field $J\xi$ of $M$ is a principal vector field with zero as its principal
curvature. Conversely, if the equality sign of  \eqref{6.2} holds, then the characteristic vector field
$J\xi$ of $M$ is a principal vector field with zero as its principal curvature only if $M=N_T\times_f
N^\perp$ is a trivial CR-warped product immersed in $\tilde M$ as a totally geodesic hypersurface.

 Also, when $M$ is anti-holomorphic with $p=1$, the equality sign of  \eqref{6.2} holds identically if and only if $M$ is a minimal hypersurface in $\tilde M$.
 \end{theorem}  

$CR$-warped products in complex space forms satisfying the equality case of \eqref{6.2} have been completely classified in \cite{c7,c8}.

\begin{theorem}\label{T:6.3} A $CR$-warped product $N_T\times_{f} N_\perp$ in  $\mathbb C^m$ satisfies
\begin{equation}||\sigma||^2= 2p||\nabla (\ln f)||^2\end{equation} identically if and only if the following four statements hold:
\begin{enumerate}
\item[(i)] $N_T$ is an open portion of a complex Euclidean $h$-space $\mathbb C^h$,

\item[(ii)] $N_\perp$ is an open portion of the unit $p$-sphere $S^p$, 

\item[(iii)] there exists  $a=(a_1,\ldots,a_h)\in S^{h-1}\subset {\mathbb E}^h$ such that $f=\sqrt{\left<a,z\right>^2 +\left<ia,z\right>^2}$ for  $z=(z_1,\ldots,z_h)\in{\mathbb C}^h,\, w=(w_0,\ldots,w_p)\in S^p\subset {\mathbb E}^{p+1}$, and

\item[(iv)] up to rigid motions, the immersion is given by
\begin{align}\notag \hbox{\bf x}(z,w)=&\Bigg(  z_1+(w_0-1)a_1\sum_{j=1}^h
a_jz_j,\,\cdots z_h+a(w_0-1)a_h\sum_{j=1}^ha_jz_j,\\& \hskip.5in w_1\sum_{j=1}^h a_jz_j,\,\,\ldots,w_p\sum_{j=1}^h a_jz_j,0,\,\ldots,0\Bigg).\notag\notag\end{align}
\end{enumerate}\end{theorem}

A $CR$-warped product $N_T\times_f N_\perp$ is said to be {\it trivial\/} if its warping function $f$ is constant. A trivial $CR$-warped product $N_T\times_f N_\perp$ is nothing but a $CR$-product $N_T\times N_\perp^f$, where $N_\perp^f$ is the  manifold  with metric $f^2g_{N_\perp}$ which is  homothetic to the original metric $g_{N_\perp}$ on $N_\perp$.

The following result  completely classifies $CR$-warped products in complex projective spaces satisfying the equality case of \eqref{6.2} identically.

\begin{theorem}\label{T:6.4} \cite{c8} A non-trivial $CR$-warped product $\,N_T\times_{f} N_\perp$   in the complex projective $m$-space $CP^m(4)$ satisfies the basic equality $||\sigma||^2= 2p||\nabla (\ln f)||^2$ if and only if we have
\begin{enumerate}
\item  $N_T$ is an open portion of complex Euclidean $h$-space ${\mathbb C}^h$,

\item $N_\perp$ is an open portion of a unit $p$-sphere $S^p$, and

\item up to rigid motions, the immersion {\bf x} of $N_T\times_{f} N_\perp$ into $CP^m$ is the composition $\pi\circ\breve{\hbox{\bf x}}$, where 
\begin{align}\notag  \breve{\hbox{\bf x}}&(z,w)=\Bigg( z_0+(w_0-1)a_0\sum_{j=0}^h a_j z_j,\,\cdots ,z_h+(w_0-1)a_h\sum_{j=0}^h a_jz_j,\\&\hskip.8inw_1\sum_{j=0}^h a_jz_j,\,\ldots,\;w_p\sum_{j=0}^h
a_jz_j,0,\,\ldots,0\Bigg),\notag\end{align} 
$\pi$ is the projection $\pi:{\mathbb C}^{m+1}_*\to CP^m$, $a_0,\ldots,a_h$ are real numbers satisfying $a_0^2+a_1^2+\cdots+a_h^2=1$, $z=(z_0,z_1,\ldots,z_h)\in {\mathbb C}^{h+1}$ and $w=(w_0,\ldots,w_p)\in S^p\subset {\mathbb E}^{p+1}$.
\end{enumerate} \end{theorem}

The following result  completely classifies  $CR$-warped products in complex hyperbolic spaces satisfying the equality case of \eqref{6.2} identically.

\begin{theorem} \label{T:6.5}   \cite{c8} A $CR$-warped product $N_T\times_{f} N_\perp$  in the complex hyperbolic $m$-space $CH^m(-4)$ satisfies the basic equality $$||\sigma||^2= 2p||\nabla (\ln f)||^2$$ if and only if one of the following two cases occurs:
\begin{enumerate}
\item $N_T$ is an open portion of complex Euclidean $h$-space ${\mathbb C}^h$, $N_\perp$ is an open portion of  a unit $p$-sphere $S^p$ and, up to rigid motions, the immersion is the composition $\pi\circ\breve{\hbox{\bf x}}$, where $\pi$ is the projection $\pi:{\mathbb C}^{m+1}_{*1}\to CH^m$ and 
\begin{align}\notag &\breve{\hbox{\bf x}}(z_,w)=\Bigg( z_0+a_0(1-w_0)\sum_{j=0}^h a_j z_j,z_1+a_1(w_0-1)\sum_{j=0}^h a_j z_j,\,\cdots, \\&\hskip.1 in  z_h+a_h(w_0-1)\sum_{j=0}^h a_jz_j,\;w_1\sum_{j=0}^h a_jz_j,\ldots,\; w_p\sum_{j=0}^h
a_jz_j,0,\ldots,0\Bigg)\notag \end{align} 
for some real numbers  $a_0,\ldots,a_{h}$  satisfying
$a_0^2-a_1^2-\cdots-a_{h}^2=-1$, where 
$z=(z_0,\ldots,z_h)\in{\mathbb C}^{h+1}_1$ and $
w=(w_0,\ldots,w_p)\in S^p\subset{\mathbb E}^{p+1}$.

\item $p=1$,   $N_T$ is an open portion of  ${\mathbb C}^h$ and, up to rigid motions, the immersion  is  the composition $\pi\circ\breve{\hbox{\bf x}}$, where 
\begin{align} \notag &\breve{\hbox{\bf x}}(z,t)=\Bigg(z_0+a_0(\cosh t-1)\sum_{j=0}^h a_j z_j,  z_1+a_1(1-\cosh t)\sum_{j=0}^h
a_jz_j,\\&\hskip.3in \ldots,z_h+a_h(1-\cosh t)\sum_{j=0}^h a_jz_j, \sinh t\sum_{j=0}^h a_jz_j,0,\ldots,0\Bigg)\notag
\end{align} 
for some real numbers $a_0,a_1\ldots,a_{h+1}$ satisfying $a_0^2-a_1^2-\cdots-a_{h}^2=1$.
\end{enumerate}\end{theorem}

A multiply warped product $N_T \times_{f_2} N_2\times \cdots \times_{f_k} N_k$ in a Kaehler manifold $\tilde M$ is called a {\it multiply $CR$-warped product}  if $N_T$ is a holomorphic submanifold and $N_\perp={}_{f_2} N_2 \times \cdots\times_{f_k} N_k$ is a totally real submanifold of $\tilde M$.

The next theorem extends \ref{T:6.2} for multiply $CR$-warped products.

\begin{theorem}\label{T:2} \cite{CD08} Let $N=N_T \times_{f_2} N_2 \times \cdots \times_{f_k} N_k$ be a multiply $CR$-warped product in an arbitrary Kaehler manifold $\tilde M$. Then the second fundamental form $\sigma$ and the warping functions $f_2,\ldots, f_k$ satisfy
 \begin{align} \label{6.4} &||\sigma||^2 \geq 2\sum_{i=2}^k n_i ||\nabla (\ln f_i)||^2. \end{align}

 The equality sign of inequality \eqref{6.4} holds identically if and only if the following four statements hold:

\begin{enumerate}
\item[(a)] $N_T$ is a  totally geodesic holomorphic submanifold of $\tilde M$;

\item[(b)]  For each $i\in \{2,\ldots,k\}$,   $N_i$ is  a totally umbilical submanifold of $\tilde M$ with $-\nabla (\ln f_i)$ as its mean
curvature vector;

\item[(c)] ${}_{f_2} N_2 \times \cdots \times_{f_k} N_k$ is immersed as mixed totally geodesic submanifold in $\tilde M$;  and

 \item[(d)]  For each point $p\in N$, the first normal space ${\rm Im} \,h_p$ is a subspace of $J(T_pN_\perp)$.
\end{enumerate} \end{theorem}

\begin{remark} B. Sahin \cite{Sahin1} extends Theorem \ref{T:6.1} to the following.

\begin{theorem} \label{T:Sahin1}There  exist no warped product submanifolds of the type $M_\theta\times_f M_T$ and $M_T \times_f M_\theta$  in a Kaehler manifold, where $M_\theta$ is a proper slant submanifold
and $M_T$ is a holomorphic submanifold of $\tilde M$.
\end{theorem}
\end{remark}

\begin{remark} As an extension of Theorem \ref{T:Sahin1} the following non-existence result was proved  by K. A. Khan, S. Ali and N. Jamal.

\begin{theorem} \cite{KAJ08} There do not exist proper warped product submanifolds of the type $N\times_f N_T$ and $N_T\times_f N$ in a Kaehler manifold, where $N_T$ 
 is a complex submanifold and $N$ is any  non-totally real generic submanifold of  a Kaehler manifold $\tilde M$.\end{theorem}
\end{remark}

\begin{remark} B. Sahin proved the following.

\begin{theorem} \label{T:Sahin2} \cite{Sahin2} There do not exist doubly warped product $CR$-submanifolds which are not (singly) warped product CR-submanifolds
in the form ${}_{f_1}M_T\times_{f_2} M_\perp$, where $M_T$ is a holomorphic submanifold
and $M_\perp$ is a totally real submanifold of a Kaehler manifold $\tilde M$.
\end{theorem}
\end{remark}

\section{$CR$-warped products with compact holomorphic factor}

When the holomorphic factor $N_{T}$ of a $CR$-warped product $N_{T}\times_{f}N_{\perp}$ is compact,  we have the following sharp results.

\begin{theorem} \label{T:7.1} \cite{c04-2} Let $N_T\times_f N_\perp$ be a $CR$-warped product in the complex
projective $m$-space $CP^m(4) $ of constant holomorphic sectional curvature $4$. If $N_T$ is compact, then we have $$m\geq h+p+hp.$$
\end{theorem}

\begin{remark} The example mentioned in Statement (c) of Theorem \ref{T:5.2} shows that Theorem \ref{T:7.1} is sharp.
\end{remark}

\begin{theorem} \label{T:7.2} \cite{c04-2} If $\,N_T\times_f N_\perp\,$ is a $CR$-warped product in $\,CP^{h+p+hp}(4)\,$ with compact $N_T$, then $N_T$ is holomorphically isometric to $CP^h$.
\end{theorem} 

\begin{theorem} \label{T:7.3} \cite{c04-2} For any $CR$-warped product $N_T\times_f N_\perp$ in $CP^m(4) $ with compact $N_T$ and any $q\in N_\perp$, we have
\begin{align}\label{7.1}\int_{N_T\times\{q\}} ||\sigma||^2dV_{T}\geq 4hp\,\hbox{\rm vol}(N_T),\end{align} 
where $||\sigma||$ is the norm of the second fundamental form,  $dV_T$ is the volume element of $N_{T}$, and  $\hbox{\rm vol}(N_T)$ is the volume of $N_T$.

The equality sign of \eqref{7.1} holds identically if and only if  we have:

\begin{enumerate}
\item The warping function $f$ is constant.

\item  $(N_T,g_{N_T})$ is holomorphically isometric to $CP^h(4)$ and it is isometrically immersed in $CP^m$ as a totally geodesic complex submanifold.

\item $(N_\perp,f^2 g_{N_\perp})$ is isometric to an open portion of the real projective
$p$-space $RP^p(1)$ of constant sectional curvature one  and it is isometrically immersed in
$CP^m$ as a totally geodesic totally real submanifold.

\item $N_T\times_f N_\perp$ is immersed linearly fully in a complex subspace $CP^{h+p+hp}(4)$ of $CP^m(4)$; and moreover, the immersion is rigid.
\end{enumerate}\end{theorem}

\begin{theorem} \label{T:7.4} \cite{c04-2} Let $N_T\times_f N_\perp$ be a $CR$-warped product with compact $N_T$ in $CP^m(4) $. If the warping function $f$ is a non-constant function, then for each $q\in N_\perp$ we have
\begin{align}\label{7.2}\int_{N_T\times\{q\}} ||\sigma||^2dV_{T}\geq 2p\lambda_1 \int_{N_T}(\ln f)^2dV_T+ 4hp\,\hbox{\rm vol}(N_T),\end{align}  
where  $\lambda_1$  is the first positive eigenvalue of the Laplacian $\Delta$ of $N_T$.

Moreover, the equality sign of \eqref{7.2} holds identically if and only if  we have 

\begin{enumerate}
\item  $\Delta\ln f=\lambda_1 \ln f$.

\item The $CR$-warped product is both $N_T$-totally geodesic and $N_\perp$-totally geodesic.
\end{enumerate}\end{theorem}

The following example shows that Theorems \ref{T:7.3} and \ref{T:7.4}  are sharp.

\begin{example} {\rm Let $\iota_1$ be the identity map of $CP^h(4)$ and let $$\iota_2:RP^p(1)\to CP^p(4)$$ be a totally geodesic Lagrangian embedding of $RP^p(1)$ into $CP^p(4)$. Denote by $$\iota=(\iota_1,\iota_2):CP^h(4)\times
RP^p(1)\to CP^h(4)\times CP^p(4)$$ the product embedding of $\iota_1$ and $\iota_2$. 
Moreover, let $S_{h,p}$ be the Segre embedding of  $CP^h(4)\times CP^p(4)$ into $CP^{hp+h+p}(4)$. Then the composition $\phi=S_{h,p}\circ\iota$:
\begin{align}  CP^h (4)&\times RP^p(1)\xrightarrow[\text{totally geodesic}]  {\text{$(\iota_1,\iota_2)$}}  CP^h(4)\times CP^p(4) \xrightarrow[\text{Segre
embedding}] {\text{$S_{h,p}$}} CP^{hp+h+p}(4)\notag\end{align} is a $CR$-warped product in $CP^{h+p+hp}(4)$ whose holomorphic factor $N_T=CP^h(4)$ is a compact manifold. 
 Since the second fundamental form of $\phi$ satisfies the equation: $||\sigma||^2=4hp$, we have the equality case of inequality \eqref{7.1} identically. }
\end{example} 
 
The next example shows that the assumption of compactness in Theorems \ref{T:7.3} and \ref{T:7.4} cannot be removed.

\begin{example}{\rm  Let ${\mathbb C}^*={\mathbb C}-\{0\}$ and ${\mathbb C}_*^{m+1}={\mathbb C}^{m+1}-\{0\}$. Denote by $\{z_0,\ldots,z_h\}$   a natural complex coordinate system on ${\mathbb C}^{m+1}_*$.  

Consider the action of
${\mathbb C}^*$ on ${\mathbb C}_*^{m+1}$ given by
$$\lambda\cdot (z_0,\ldots,z_m)=(\lambda z_0,\ldots,\lambda z_m)$$ for $\lambda\in {\mathbb C}_*$.  Let
$\pi(z)$ denote the equivalent class containing $z$ under this action. Then the  set of equivalent classes is the complex projective $m$-space $CP^m(4)$ with the  complex structure induced from the complex structure on ${\mathbb C}^{m+1}_*$. 

For any two natural numbers $h$ and $p$, we define a map: $$\breve \phi:\mathbb C^{h+1}_*\times S^p(1)\to\mathbb C^{h+p+1}_*$$ by
$$\breve \phi(z_0,\ldots,z_h;w_0,\ldots,w_p)=\big(w_0z_0,w_1z_0,\ldots,w_pz_0,z_1,\ldots,z_h\big)$$ 
for $(z_0,\ldots,z_h)$ in $\mathbb C^{h+1}_*$ and $(w_0,\ldots,w_p)$ in $S^{p}$ with $\sum_{j=0}^p w_j^2=1$.

Since the image of $\breve \phi$ is invariant under the action of ${\mathbb C}_*$, the composition:
\begin{align}\pi\circ\breve \phi: {\mathbb C}^{h+1}_*\times S^p\xrightarrow[\text{}] {\text{$\breve \phi$}} {\mathbb C}^{h+p+1}_*\xrightarrow[\text{}] {\text{$\pi$}} CP^{h+p}(4)\notag\end{align} 
induces a $CR$-immersion of the product manifold $N_T\times S^p$ into $CP^{h+p}(4)$, where $$N_T=\big\{(z_0,\ldots,z_h)\in CP^h:z_0\ne 0\big\}$$ is a proper open subset of $CP^h(4)$. Clearly, the induced metric on
$N_T\times S^p$ is a warped product metric and the holomorphic factor $N_T$ is non-compact. 

Notice that the complex dimension of the ambient space is  $h+p$; far less than $h+p+hp$.}
\end{example}

\section{{Another optimal inequality for $CR$-warped products}}  

All $CR$-warped products in complex space forms also satisfy another general optimal inequality obtained in \cite{c03}. 

\begin{theorem} \label{T:8.1} Let $N=N_T^h\times_f N_\perp^p$ be a  $CR$-warped product in a complex space form $\tilde M(4c)$ of constant holomorphic sectional curvature $c$. Then we have
\begin{equation}\label{8.1}||\sigma||^2\geq 2p\big\{||\nabla (\ln f)||^2+\Delta(\ln f)+2hc\big\}.\end{equation}
  
If the equality sign of \eqref{8.1}  holds identically, then $N_T$ is a totally geodesic submanifold and $N_\perp$ is a totally umbilical submanifold. Moreover, $N$ is a minimal submanifold in $\tilde M(4c)$. \end{theorem} 

The following three theorems  completely classify all $CR$-warped products which satisfy the equality case of \eqref{8.1} identically.

 \begin{theorem}\label{T:8.2} \cite{c03}  Let
$\phi:N_T^h\times_f N_\perp^p\to {\mathbb C}^m$ be a  $CR$-warped product in  $\mathbb C^m$. Then we have
\begin{equation}\label{8.2}||\sigma||^2\geq 2p\big\{||\nabla(\ln f)||^2+\Delta(\ln f)\big\}.\end{equation}

 The  equality case of inequality  \eqref{8.2} holds identically if and only if the following four statements hold.
\begin{enumerate}
\item $N_T$ is an open portion of $\mathbb  C^h_*:=\mathbb  C^h-\{0\}$;

\item $N_\perp$ is an open portion of $S^p$;

\item There is $\alpha,\, 1\leq\alpha\leq h,$ and complex Euclidean  coordinates $\{z_1,\ldots,z_h\}$ on $\mathbb C^h$ such that $f=\sqrt{\sum_{j=1}^\alpha z_j\bar z_j}$\,; 

\item Up to rigid motions, the immersion $\phi$ is given by
\begin{equation}\notag\phi=\big(w_0 z_1,\ldots,w_p z_1,\ldots,w_0z_\alpha ,\ldots, w_p z_\alpha ,z_{\alpha
+1},\ldots,z_h,0,\ldots,0\big)\end{equation} for $z=(z_1,\ldots,z_h)\in {\mathbb C}_*^h$ and $\,w=(w_0,\ldots,w_p)\in S^p(1)\subset{\mathbb E}^{p+1}$.
\end{enumerate} \end{theorem} 

 \begin{theorem} \label{T:8.3}  \cite{c03} Let $\phi:N_T\times_{f} N_\perp\to CP^m(4)$  be a $CR$-warped product  with $\dim_{\mathbb C} N_T=h$ and $\dim_{\mathbb R} N_\perp=p$. Then we have
\begin{equation}\label{8.3}||\sigma||^2\geq 2p\big\{||\nabla (\ln f)||^2+\Delta(\ln f)+2h\big\}.\end{equation}

 The $CR$-warped product  satisfies the equality case of inequality \eqref{8.3} identically if and only if the following three statements hold.

\begin{enumerate}
\item[(a)]  $N_T$ is an open portion of  complex projective $h$-space $CP^h(4)$;

\item[(b)]   $N_\perp$ is an open portion of  unit $p$-sphere $S^p(1)$; and

\item[(c)]  There exists a natural number $\alpha\leq h$ such that, up to rigid motions,  $\phi$ is the composition $\pi\circ\breve{\phi}$, where 
\begin{equation}\begin{aligned} &\notag \breve{\phi}(z,w)=\big(w_0 z_0,\ldots,w_p
z_0,\ldots,w_0z_\alpha ,\ldots, w_p z_\alpha ,z_{\alpha +1},\ldots,z_h,\\&\hskip1.5in 0\ldots,0\big)\end{aligned}\end{equation} for $z=(z_0,\ldots,z_h)\in
{\mathbb C}_*^{h+1}$ and $\,w=(w_0,\ldots,w_p)\in
S^p(1)\subset{\mathbb E}^{p+1}$, where  $\pi$ is the projection $\pi:{\mathbb C}^{m+1}_*\to CP^m$.
\end{enumerate} \end{theorem}

 \begin{theorem} \label{T:8.4} \cite{c03} Let $\phi:N_T\times_{f} N_\perp\to CH^m(-4)$  be a $CR$-warped product  with
$\dim_{{\mathbb C}} N_T=h$ and $\dim_{{\mathbb R}} N_\perp=p$. Then we have
\begin{equation}\label{8.4}||\sigma||^2\geq 2p\big\{||\nabla (\ln f)||^2+\Delta(\ln f)-2h\big\}.\end{equation}

 The $CR$-warped product satisfies the equality case of \eqref{8.4} identically if and only if the following three statements hold.

\begin{enumerate}
\item[(a)]   $N_T$ is an open portion of  complex hyperbolic $h$-space $CH^h(-4)$;

\item[(b)]  $N_\perp$ is an open portion of  unit $p$-sphere $S^p(1)$ $($or {$\mathbb R$}, when $p=1$$)$; and

\item[(c)]  up to rigid motions, $\phi$ is the composition $\pi\circ\breve{\phi}$, where either $\breve\phi$ is  given by
\begin{equation}\begin{aligned} &\notag\breve{\phi}(z,w)=\big(z_{0},\ldots,z_{\beta},w_0 z_{\beta+1},\ldots,w_p z_{\beta+1},\ldots,w_0z_h ,\ldots, \\& \hskip1.4in w_p z_h ,0\ldots,0\big)\end{aligned}\end{equation} for  $0<\beta\leq h$,  $z=(z_0,\ldots,z_h)\in {\mathbb C}_{*1}^{h+1}$ and $\,w=(w_0,\ldots,w_p)\in S^p(1)$, or $\breve \phi$ is given by
\begin{align} \notag&\breve{\phi}(z,u)=\big(z_0\cosh u,z_0\sinh u,z_1\cos
u,z_1\sin u,\ldots, \\&\hskip.9in z_\alpha\cos u,z_\alpha\sin u,z_{\alpha+1},\ldots,z_h,0\ldots,0\big) \notag\end{align}
 for $z=(z_0,\ldots,z_h)\in {\mathbb C}_{*1}^{h+1}$, where  $\pi$ is the projection $\pi:{\mathbb C}^{m+1}_{*1}\to CH^m(-4)$.
\end{enumerate} \end{theorem}

\section{{Warped product real hypersurfaces in complex space forms}}

For real hypersurfaces,  we have  the following non-existence theorem.

\begin{theorem}\label{T:9.1}  \cite{CM} There do not exist real hypersurfaces in complex projective and complex hyperbolic spaces which are Riemannian products of two or more Riemannian manifolds of positive dimension. 

In other words, every real hypersurface in a non-flat complex space form is irreducible.
\end{theorem}

A contact manifold is an odd-dimensional manifold $M^{2n+1}$ with  a 1-form $\eta$ such that $\eta\wedge(d\eta)^n\not=0$.  A curve $\gamma=\gamma(t)$ in a contact manifold is called a {\it Legendre curve} if $\eta(\beta'(t))=0$ along $\beta$. Let $S^{2n+1}(c)$ denote the hypersphere in ${\mathbb C}^{n+1}$ with curvature  $c$ centered at the origin. Then $S^{2n+1}(c)$ is a contact manifold endowed with a canonical contact structure which is the dual 1-form of the characteristic vector field $J\xi$, where $J$ is the complex structure and $\xi$ the unit normal vector on $S^{2n+1}(c)$.

 Legendre curves are known to play an important role in the study of contact manifolds, e.g. a diffeomorphism of a contact manifold is a contact transformation if and only if it maps Legendre curves to Legendre curves.

Contrast to Theorem \ref{T:9.1}, there exist many warped product real hypersurfaces in complex space forms as given in the following three theorems from \cite{c02-3}.

\begin{theorem}\label{T:9.2}  Let $a$ be a positive number and $\gamma(t)=(\Gamma_1(t),\Gamma_2(t))$ be a unit speed Legendre curve $\gamma:I\to S^3(a^2)\subset{\mathbb C}^2$ defined on an open interval  $I$.  Then
\begin{align} &\hbox{\bf x}(z_1,\ldots,z_{n},t)=\big(a\Gamma_1(t)z_1, a\Gamma_2(t) z_1,z_2,\ldots,z_n\big),\quad z_1\ne 0\end{align} 
defines a real hypersurface which is the warped product ${\mathbb C}_{**}^{n}\times_{a|z_1|} I$ of a complex $n$-plane and $I$, where ${\mathbb C}_{**}^{n}=\{(z_1,\ldots,z_n):z_1\ne 0\}$.

Conversely, up to rigid motions of ${\mathbb C}^{n+1}$, every  real hypersurface in ${\mathbb C}^{n+1}$ which is the warped product $N\times_f I$ of a complex hypersurface $N$  and an open interval $I$ is either obtained in the way described above or given by the product submanifold ${\mathbb C}^n\times C\subset {\mathbb C}^{n}\times{\mathbb C}^1$ of ${\mathbb C}^{n}$ and a real curve $C$ in ${\mathbb C}^1$.
\end{theorem}

 Let $S^{2n+3}(1)$ denote the unit hypersphere in  ${\mathbb C}^{n+2}$ centered at the origin and put $$U(1)=\{\lambda\in {\mathbb C}:\lambda\bar\lambda=1\}.$$ Then there is a $U(1)$-action on $S^{2n+3}(1)$ defined by $z\mapsto \lambda z$. At $z\in S^{2n+3}(1)$ the vector $V=iz$ is tangent to the flow of the action. The quotient space $S^{2n+3}(1)/\sim$, under the identification induced from the action, is a complex projective space $ CP^{n+1}$ which endows with the canonical Fubini-Study metric of constant holomorphic sectional curvature $4$. 
 
 The almost complex structure $J$ on $ CP^{n+1}(4)$ is induced from the complex structure $J$ on ${\mathbb C}^{n+2}$ via the Hopf fibration: $\,\pi : S^{2n+3}(1)\to CP^{n+1}(4)$. It is well-known that the Hopf fibration $\pi$ is a Riemannian submersion such that $V=iz$ spans the vertical subspaces. 
Let $\phi:M\to CP^{n+1}(4)$ be an isometric immersion. Then $\hat M=\pi^{-1}(M)$ is a principal circle bundle over $M$ with totally geodesic fibers. The lift $\hat \phi:\hat M\to S^{2n+3}(1)$ of $\phi$ is an isometric immersion so that the diagram:
\[\label{E:5.b} \begin{CD} \hat M @>\hat \phi >>S^{2n+3}(1)\\ @V{\pi}VV @VV{\pi}V\\ M @> \phi>>  CP^{n+1}(4)\end{CD} \]
commutes.
 
Conversely, if $\psi:\hat M \to S^{2n+3}(1)$ is an isometric immersion which is invariant under the $U(1)$-action, then there is a unique isometric immersion $\psi_\pi:\pi(\hat M)\to CP^{n+1}(4)$ such that the associated diagram commutes. We simply call the immersion $\psi_\pi:\pi(\hat M)\to CP^{n+1}(4)$ {\it the projection} of $\psi:\hat M\to S^{2n+3}(1)$.

For a given  vector $X\in T_z(CP^{n+1})(4)$ and a point $u\in S^{2n+2}(1)$ with $\pi(u)=z$, we denote by $X_u^*$ the horizontal lift of $X$ at $u$ via $\pi.$ There exists  a canonical orthogonal decomposition:
\begin{equation}T_u S^{2n+3}(1)=(T_{\pi(u)}CP^{n+1}(4))_u^*\oplus
\hbox{\rm Span}\,\{V_u\}.\end{equation}
Since $\pi$ is a Riemannian  submersion, $X$ and  $X^*_u$ have the same length.

We put
\begin{equation}\begin{aligned}&S_*^{2n+1}(1)=\left\{(z_0,\ldots,z_n):\sum_{k=0}^n
z_k\bar z_k=1,\, z_0\ne 0\right\}, \\& CP^n_0=\pi(S_*^{2n+1}(1)).\end{aligned}\end{equation}

The next theorem classifies all warped products hypersurfaces of the form $N\times_{f}I$ in complex projective spaces.

\begin{theorem}\label{T:9.3}   \cite{c02-3} Suppose that $a$ is a
positive number and $\gamma(t)=(\Gamma_1(t),\Gamma_2(t))$ is a
unit speed Legendre curve $\gamma:I\to S^3(a^2)\subset{\mathbb C}^2$ defined on an open interval $I$.  Let $\hbox{\bf x}:S_*^{2n+1}\times I\to{\mathbb C}^{n+2}$ be the map defined by
\begin{align} &\hbox{\bf x}(z_0,\ldots,z_{n},t)=\big(a\Gamma_1(t)z_0,a\Gamma_2(t) z_0,z_1,\ldots,z_{n}\big),\;\;\;\sum_{k=0}^{n} z_k\bar z_k=1.\end{align} 
Then 
\begin{enumerate}
\item $\,\hbox{\bf x}$ induces an isometric immersion $\psi:S_*^{2n+1}(1)\times_{a|z_0|} I\to S^{2n+3}(1)$.

\item  The image $\psi(S_*^{2n+1}(1)\times_{a|z_0|} I)$ in $S^{2n+3}(1)$ is invariant under the action of $U(1)$.

\item the projection $$\psi_\pi:\pi(S_*^{2n+1}(1)\times_{a|z_0|} I) \to CP^{n+1}(4)$$ of $\psi$ via $\pi$ is a warped product hypersurface $CP^{n}_0\times_{a|z_0|} I$ in $CP^{n+1}(4)$.
\end{enumerate}

Conversely, if a  real  hypersurface in $CP^{n+1}(4)$ is a warped product $N\times_f I$ of a complex hypersurface $N$ of $CP^{n+1}(4)$  and an open interval $I$, then, up to rigid motions, it is locally obtained in the way described above.
\end{theorem}

In the complex pseudo-Euclidean space ${\mathbb C}^{n+2}_1$ endowed with  pseudo-Euclidean metric
\begin{equation}g_0=-dz_0d\bar z_0 +\sum_{j=1}^{n+1}dz_jd\bar z_j,\end{equation} we  define the anti-de Sitter space-time  by
\begin{equation}H^{2n+3}_1(-1)=\big\{(z_0,z_1,\ldots,z_{n+1}):
\left<z,z\right>=-1\big\}.\end{equation} It is known that $H^{2n+3}_1(-1)$ has constant sectional curvature $-1$. There is a $U(1)$-action on
$H_1^{2n+3}(-1)$ defined by $z\mapsto \lambda z$. At a point $z\in H^{2n+3}_1(-1)$,  $iz$ is tangent to the flow of the action. The orbit is given by  $z_t=e^{it}z$ with ${{dz_t}\over{dt}}=iz_t$ which lies in the negative-definite plane spanned by $z$ and $iz$. 

The quotient space $H^{2n+3}_1(-1)/\sim$ is the complex hyperbolic space $CH^{n+1}(-4)$ which endows a canonical Kaehler metric of constant holomorphic sectional curvature $-4$. The  complex structure $J$ on $ CH^{n+1}(-4)$ is induced from the canonical complex structure $J$ on $ {\mathbb C}^{n+2}_1$ via the totally geodesic fibration: $\pi:H^{2n+3}_1\rightarrow  CH^{n+1}(-4).$

Let $\phi:M\to CH^{n+1}(-4)$ be an isometric immersion. Then $\hat M=\pi^{-1}(M)$ is a principal circle bundle over $M$ with totally geodesic fibers. The lift $\hat \phi:\hat M\to H_1^{2n+3}(-1)$ of $\phi$ is an isometric immersion such that the
 diagram:
\[\begin{CD} \hat M @>\hat \phi >>H_1^{2n+3}(-1)\\ @V{\pi}VV @VV{\pi}V\\ M @> \phi>>  CH^{n+1}(-4)\end{CD}\]
commutes.

Conversely, if $\psi:\hat M \to H_1^{2n+3}(-1)$ is an isometric immersion which is invariant under the $U(1)$-action, there
is a unique isometric immersion $\psi_\pi:\pi(\hat M)\to CH^{n+1}(-4)$,  called the {\it projection of\/} $\psi$ so that the associated diagram commutes.

We put
\begin{align} & H_{1*}^{2n+1}(-1)=\big\{(z_0,\ldots,z_{n})\in H^{2n+1}_1(-1): z_n\ne 0\big\},\\& CH^n_*(-4)=\pi(H^{2n+1}_{1*}(-1)).\end{align}

The following theorem  classifies all warped products hypersurfaces of the form $N\times_{f}I$ in complex hyperbolic spaces.

 \begin{theorem}\label{T:9.4} \cite{c02-3} Suppose that $a$ is a positive number and $\gamma(t)=(\Gamma_1(t),\Gamma_2(t))$ is a unit speed Legendre curve $\gamma:I\to S^3(a^2)\subset{\mathbb C}^2$.  Let $$\hbox{\bf y}:H_{1*}^{2n+1}(-1)\times I\to{\mathbb C}^{n+2}_1$$ be the map defined by
\begin{align} &\hbox{\bf y}(z_0,\ldots,z_{n},t)=(z_0,\ldots,z_{n-1},
a\Gamma_1(t)z_{n},a\Gamma_2(t) z_{n}),\\&\hskip 1in z_0\bar z_0-\sum_{k=1}^{n} z_k\bar z_k=1.\end{align}
Then we have
\begin{enumerate}
\item  $\,\hbox{\bf y}$  induces an isometric immersion $\psi:H_{1*}^{2n+1}(-1)\times_{a|z_n|} I\to H_1^{2n+3}(-1)$.

\item  The image $\psi(H_{1*}^{2n+1}(-1)\times_{a|z_n|} I)$ in $H_1^{2n+3}(-1)$ is invariant under the $U(1)$-action.

\item the projection $$\psi_\pi:\pi(H_{1*}^{2n+1}(-1)\times_{a|z_n|} I) \to CH^{n+1}(-4)$$ of $\psi$ via $\pi$ is a warped product hypersurface $CH^{n}_*(-1)\times_{a|z_n|} I$ in $CH^{n+1}(-4)$.
\end{enumerate}

Conversely, if a  real hypersurface in $CH^{n+1}(-4)$ is a warped product $N\times_f I$ of a complex hypersurface $N$ and an open interval $I$, then, up to rigid motions, it is locally obtained in the way described above.
\end{theorem}

\section{{Twisted product $CR$-submanifolds of Kaehler manifolds}}

Twisted products $B\times_{\lambda} F$ are natural extensions of warped products, namely the function may depend on both factors (cf. \cite[page 66]{cbook2}). 
When $\lambda$ depends only on $B$, the twisted product becomes a  warped product.  If $B$ is a point, the twisted product is nothing but a conformal change of metric on $F$.

The study of twisted product $CR$-submanifolds was initiated by the author in 2000 (see \cite{c5.1}). In particular, the following results are obtained.

\begin{theorem}\label{T:10.1} \cite{c5.1} If  $\,M=N_\perp\times_\lambda N_T$ is a twisted product $CR$-submanifold of a Kaehler manifold $\tilde M$ such that $N_\perp$ is a totally real submanifold and $N_T$ is a holomorphic submanifold of $\tilde M$, then $M$ is a $CR$-product. \end{theorem}

\begin{theorem}\label{T:10.2} \cite{c5.1} Let $M=N_T\times_\lambda N_\perp$ be a  twisted product $CR$-submanifold of a Kaehler manifold $\tilde M$ such that $N_\perp$ is a totally real submanifold
and $N_T$ is a holomorphic submanifold of $\tilde M$. Then we have

\begin{enumerate}
\item The squared norm of the second fundamental form of $M$ in $\tilde M$ satisfies
\begin{equation}\notag ||\sigma||^2\geq 2\, p\, ||\nabla^T(\ln\lambda)||^2,\end{equation}
where $\nabla^T(\ln\lambda)$ is the $N^T$-component of the gradient $\nabla(\ln \lambda)$ of $\,\ln\lambda$ and $p$ is the dimension of $N_\perp$.

\item If $\,||\sigma||^2= 2p\,||\nabla^T\ln\lambda ||^2\,$ holds identically, then $N_T$ is a
totally geodesic submanifold and $N_\perp$ is a totally umbilical  submanifold of $\tilde M$.

\item If $M$ is anti-holomorphic in $\tilde M$ and $\dim N_\perp>1$, then  $\,||\sigma||^2=
2p\,{||\nabla^T\ln\lambda ||^2}\,$ holds identically if and only if $N_T$ is a
totally geodesic  submanifold and $N_\perp$ is a totally umbilical submanifold of $\tilde M$.
\end{enumerate}\end{theorem}

 For mixed foliate twisted product $CR$-submanifolds of Kaehler manifolds, we have the following result.

\begin{theorem}\label{10.3}  \cite{c5.1} Let $M=N_T\times_\lambda N_\perp$ be a twisted product $CR$-submanifold of a Kaehler manifold $\tilde M$ such that $N_\perp$ is a totally real submanifold and $N_T$ is a holomorphic submanifold of $\tilde M$. If $M$ is mixed totally geodesic, then we have

\begin{enumerate}
\item The twisted function $\lambda$ is a function depending only on $N_\perp$. 

\item  $N_T\times N^\lambda_\perp$ is a $CR$-product, where $N^\lambda_\perp$ denotes the
manifold $N_\perp$ equipped with the metric $g_{N_\perp}^\lambda=\lambda^2 g_{N_\perp}$.
\end{enumerate}\end{theorem}

Next, we provide ample examples of  twisted product $CR$-submanifolds in complex Euclidean spaces which are not $CR$-warped product submanifolds.

Let $z:N_T\to{\mathbb C}^m$ be a holomorphic submanifold of a complex Euclidean $m$-space ${\mathbb C}^m$ and $w:N^1_\perp\to {\mathbb C}^\ell$ be a totally real submanifold such that the image of $N_T\times N^1_\perp$ under the product immersion $\psi=(z,w)$ does not contain the origin $(0,0)$ of
${\mathbb C}^m\oplus {\mathbb C}^\ell$. 

Let $j:N_\perp^2\to S^{q-1}\subset {\mathbb E}^{q}$ be an isometric immersion of a Riemannian manifold $N_\perp^2$ into the unit hypersphere $S^{q-1}$ of ${\mathbb E}^{q}$ centered at the origin.  

Consider the map
$$\phi=(z, w)\otimes j:N_T\times N^1_\perp\times N^2_\perp\to ({\mathbb C}^m\oplus {\mathbb C}^\ell)\otimes 
{\mathbb E}^{q}$$
 defined by
\begin{equation}\label{Phi}\phi(p_1,p_2,p_3)=(z(p_1),z(p_2)) \otimes j(p_3),\end{equation}
for $ p_1\in N_T, \; p_2\in N^1_\perp,\; p_3\in N^2_\perp$.

On $({\mathbb C}^m\oplus {\mathbb C}^\ell)\otimes {\mathbb E}^{q}$ we define a complex structure $J$ by
$$J((B,E)\otimes F)=({\rm i}B,{\rm i} E)\otimes F,\quad {\rm i}=\sqrt{-1},$$ for any $B\in {\mathbb C}^m,\,E\in {\mathbb C}^\ell$ and $F\in {\mathbb E}^{q}$. Then $({\mathbb C}^m\oplus {\mathbb C} ^\ell)\otimes {\mathbb E}^{q}$ becomes a complex
Euclidean $(m+\ell)q$-space ${\mathbb C}^{(m+\ell)q}$.

Let us put $N_\perp=N^1_\perp\times N^2_\perp$. We denote by $|z|$ the distance function from the origin of ${\mathbb C}^m$ to the position of $N_T$ in ${\mathbb C}^m$ via $z$; and denote by  $|w|$
the distance function from the origin of ${\mathbb C}^\ell$ to the position of $N^1_\perp$ in ${\mathbb C}^\ell$ via $w$. We define a function $\lambda$ by $\lambda=\sqrt{|z|^2+|w|^2}$.
Then $\lambda>0$ is a differentiable function on $N_T\times N_\perp$, which depends on both $N_T$ and $N_\perp=N^1_\perp\times N^2_\perp$. 

Let $M$ denote the twisted product $N_T\times_\lambda N_\perp$ with twisted function $\lambda$. Clearly, $M$  is not a warped product.

For such a twisted product $N_T\times_\lambda N_\perp$ in ${\mathbb C}^{(m+\ell)q}$ defined above we have the following.

\begin{proposition} \label{P:10.1}   \cite{c5.1} The map
$\phi=(z, w)\otimes j:N_T\times_\lambda N_\perp\to {\mathbb C}^{(m+\ell)q} $ 
defined by \eqref{Phi} satisfies the following properties:

\begin{enumerate}
\item  $\phi=(z, w)\otimes j:N_T\times_\lambda N_\perp\to {\mathbb C}^{(m+\ell)q}$  is an isometric immersion.

\item  $\phi=(z, w)\otimes j:N_T\times_\lambda N_\perp\to
{\mathbb C}^{(m+\ell)q}$ is a twisted product $CR$-submanifold such that $N_T$ is a holomorphic submanifold and $N_\perp$ is a totally real submanifold of ${\mathbb C}^{(m+\ell)q}$.
\end{enumerate}\end{proposition}

Proposition \ref{P:10.1} shows that there are many twisted product $CR$-submanifolds $N_T\times_\lambda N_\perp$ such that $N_T$ are holomorphic submanifolds and $N_\perp$ are totally real submanifolds. Moreover, such twisted product $CR$-submanifolds are not warped product $CR$-submanifolds.

Let $(B,g_B)$ and $(F,g_F)$ be Riemannian manifolds and let $\pi_B: B\times F\to B$ and $\pi_F: B\times F\to F$ be the canonical projections. Also let $, f$ be smooth real-valued functions on $B\times F$. Then the {\it doubly twisted product} of $(B,g_B)$ and $(F,g_F)$
 with twisting functions $b$ and $f$ is defined to be the product manifold
$M=B\times F$ with metric tensor 
$$g=f^2 g_B +b^2 g_F.$$ We denote this kind manifolds
by ${}_f B\times_b F$.

\begin{remark} B. Sahin proved the following.

\begin{theorem} \label{T:Sahin3}  \cite{Sahin2} There do not exist doubly twisted product $CR$-submanifolds in a Kaehler manifold which are not (singly) twisted product CR-submanifolds
in the form ${}_{f_1}M_T\times_{f_2} M_\perp$, where $M_T$ is a holomorphic submanifold
and $M_\perp$ is a totally real submanifold of the Kaehler manifold $\tilde M$.
\end{theorem}
\end{remark} 

An almost Hermitian manifold $(M,g,J)$ with  almost complex structure $J$ is called a  {\it nearly Kaehler} manifold provided that \cite{gray70} \begin{align}\label{19.1}(\nabla_XJ )X=0, \;\; \forall X\in TM.\end{align}

\begin{remark} Theorem \ref{T:Sahin3} was extended by S. Uddin  in \cite{Uddin} to doubly twisted product $CR$-submanifolds in a nearly Kaehler manifold.\end{remark}

\section{Warped products with a purely real factor in Kaehler manifolds}

Recall from \cite{c1981} that a submanifold $N$ of an almost Hermitian manifold $\tilde M$ is called  {\it purely real} if the complex structure $J$ on $\tilde M$ carries the tangent bundle of $N$ into a transversal  bundle, i.e.,  $J(TN)\cap TN=\{0\}$.

  The next non-existence result is an extension of Theorem 10.4  of \cite{Sahin1}.

\begin{theorem} {\rm \cite{AK08-2,KAJ08}} There do not exist non-trivial warped product submanifolds $N^{0} \times_f N^{T}$ in a Kaehler manifold $\tilde M$ such that
 $N^{T}$ is a complex submanifold and $N^{0}$ is a proper purely real submanifold of $\tilde M$. \end{theorem}

A {\it hemi-slant submanifold} of an almost Hermitian manifold is a submanifold $M$ with two orthogonal distributions $\mathcal H^\perp$ and $\mathcal H^\theta$ such that $TM= {\mathcal H}^\perp\oplus {\mathcal H}^\theta$, $J{\mathcal H}^\perp\subset T^\perp M$ and ${\mathcal H}^\theta$ is $\theta$-slant. 

\begin{remark} Hemi-slant submanifolds were first defined  in  \cite{carr00} under the name of anti-slant submanifolds.\end{remark}

The next non-existence result was proved in \cite{sahin09a}.

\begin{theorem}\label{T:11.2} There exist no warped product hemi-slant submanifolds $N^{\perp} \times_f N^{\theta}$ in a Kaehler manifold$\tilde{M}$ such that
 $N^{\perp}$ is a totally real submanifold and $N^{\theta}$ is a  proper slant submanifold of $\tilde{M}$. \end{theorem}
 
For warped product submanifolds $N^{\theta} \times_f N^{\perp}$, we also have the following result from \cite{sahin09a}.

\begin{theorem}\label{T:11.3} Let $N$ be an $(m+n)$-dimensional mixed geodesic warped product  submanifold  $N^{\theta}\times_f N^{\perp}$ in a Kaehler manifold $\tilde{M}^{m+n}$ such that $N^{\theta}$ is a proper slant submanifold and $N^{\perp}$ be a totally real submanifold of $\tilde{M}^{m+n}$. Then we have:
\begin{enumerate}
  \item [{\rm (1)}]  The second fundamental form $\sigma$ of $N$  satisfies
  \begin{equation}\label{11.1}
  ||\sigma||^2 \geq m\, (\cot^2\theta) | \nabla (\ln f)|^2, \;\; m=\dim N^{\perp}.   \end{equation}
  
  \item [{\rm (2)}] If the equality  of \e{11.1}  holds identically, then $N^{\theta}$  is a totally geodesic submanifold and $N^\perp$  is a totally umbilical submanifold of $\tilde{M}$.
Moreover, $M$ is never a minimal submanifold of $\tilde{M}$. 
\end{enumerate}
\end{theorem}

 A submanifold $N$ of an almost Hermitian manifold $\tilde M$ is called {\it semi-slant}  \cite{pap94}. if the tangent bundle $TN$ is the direct sum ${\mathcal H}\oplus {\mathcal H^\theta}$ of two orthogonal distributions ${\mathcal H}$ and ${\mathcal H}^\theta$, where ${\mathcal H}$ is holomorphic, i.e. ${\mathcal H}$ is invariant with respect to the complex structure $J$ of $\tilde M$ and ${\mathcal H}^\theta$ is a $\theta$-slant distribution\index{general}{Distribution!slant}, i.e., the angle $\theta$ between $JX$ and ${\mathcal H}^\theta_x$ is constant for any unit vector $X\in {\mathcal H}^\theta_x$ and for any point $x\in N$.

Contrast to $CR$-warped products $N^T\times_f N^\perp$, we have the following non-existence result for semi-slant warped product submanifolds. 

\begin{theorem}\label{T:11.4} {\rm  \cite{Sahin1}} There do not exist warped product submanifolds of the form: $N^T \times_f N^\theta$ in a Kaehler  manifold $\tilde M$ such that $N^T$ is a complex submanifold and $N^\theta$ is a proper slant submanifold  of $\tilde M$.
\end{theorem}

A natural extension of semi-slant submanifolds are pointwise semi-slant submanifolds defined and studied in \cite{sahin13d}.

 A submanifold $N$ of an almost Hermitian manifold $\tilde M$ is called {\it pointwise semi-slant} if the tangent bundle $TN$ is the direct sum ${\mathcal H}\oplus {\mathcal H^\theta}$ of two orthogonal distributions ${\mathcal H}$ and ${\mathcal H}^\theta$, where ${\mathcal H}$ is a holomorphic distribution and ${\mathcal H}^\theta$ is a pointwise slant distribution, i.e., for any given point $x\in N$ the angle $\theta(x)$ between $JX$ and ${\mathcal H}^\theta_x$ is independent of the choice of the unit vector $X\in {\mathcal H}^\theta_x$.

For warped product pointwise semi-slant  submanifolds, we have the following.

\begin{theorem}\label{T:11.5} {\rm \cite{sahin13d}} Let $N=N^T\times_f N^{\theta}$ be a non-trivial warped product pointwise semi-slant  submanifold of a Kaehler manifold $\tilde M^{h+p}$, where $N^{T}$ is a complex submanifold with $\dim_{\bf C} N^T=h$ and $N^{\theta}$ is a proper pointwise slant submanifold with $\dim N^\theta=p$. Then we have
\begin{enumerate}
  \item [{\rm (1)}] The second fundamental form $\sigma$ of $N$ satisfies
  \begin{equation}\label{11.80} ||\sigma||^2 \geq 2p (\csc^2\theta+\cot^2\theta) |\nabla^T (\ln f)|^2.
  \end{equation}
  \item [{\rm (2)}] If the equality  of \e{11.80}  holds identically, then $N^T$  is a totally geodesic complex submanifold and $N^{\theta}$  is a totally umbilical submanifold of $\tilde M^{h+p}$.
Moreover, $N$ is a minimal submanifold of $\tilde M^{h+p}$. 
\end{enumerate}
\end{theorem}

Pointwise bi-slant immersions are defined as follows (cf. \cite{CU}).

 A submanifold $N$ of an almost Hermitian manifold $\tilde M$ is called {\it{pointwise bi-slant}}\index{general}{Submanifold!pointwise bi-slant} if there exists a pair of orthogonal distributions ${\mathcal{H}}_1$ and ${\mathcal{H}}_2$  of $M$ such that the following three conditions hold.
\begin{enumerate}
\item[{(a)}] $TN= {\mathcal{H}}_1\oplus{\mathcal{H}}_2$;
\item[{(b)}] $J{\mathcal{H}}_1 \perp {\mathcal{H}}_2$ and $J{\mathcal{H}}_2 \perp {\mathcal{H}}_1$;

\item[{(c)}] Each distribution  ${\mathcal{H}}_i$ is pointwise slant with slant function $\theta_i\, (i=1,2)$.
\end{enumerate}
 
 \noindent   A pointwise bi-slant submanifold is a {\it bi-slant submanifold} if both slant functions $\theta_1$ and $\theta_2$ are constant.

 Analogous to $CR$-warped products, Chen and  Uddin defined warped product pointwise bi-slant submanifolds in \cite{CU} as follows. 

A warped product $N_1\times_f N_2$ in an almost Hermitian manifold $\tilde M$ is called a {\it warped product pointwise bi-slant submanifold} if both factors $N_1$ and $N_2$ are pointwise slant submanifolds of $\tilde M$.
A warped product pointwise bi-slant submanifold $N_1\times_fN_2$ is called {\it warped product bi-slant} if both $N_1$ and $N_2$ are slant submanifolds. 

For warped product bi-slant submanifolds we have the following classification result. 

\begin{theorem}\label{T:11.6}  {\rm \cite{CU}} Let $M=M_{\theta_1}\times_fM_{\theta_2}$ be a warped product bi-slant submanifold  in a Kaehler manifold $\tilde M$. Then one of the following two cases must occurs:
\begin{enumerate}
\item[{\rm (i)}]  The warping function $f$ is constant, i.e., $M$ is a Riemannian product of two slant submanifolds;

\item[{\rm (ii)}]  The slant angle of $M_{\theta_2}$ satisfies $\theta_2={\pi}/{2}$, i.e., $M$ is a warped product hemi-slant submanifold with $M_{\theta_2}$ being a  totally real submanifold $M_\perp$ of $\tilde M$.
\end{enumerate}
\end{theorem}  

For  pointwise pseudo-slant warped product submanifolds in a Kaehler Manifolds, we have the following two results from \cite{SS}.

\begin{theorem}\label{T:11.7}  Let $M=F\times_fM_{\theta}$ be a mixed geodesic warped product submanifold of a 2m-dimensional Kaehler manifold $\tilde M$ such that $F$ is a $\gamma$-dimensional totally real submanifold and $M_{\theta}$ is a $\beta$-dimensional proper pointwise slant submanifold of $M$.  Then the squared norm of the second fundamental form $||\sigma||^2$ of $M$ satisfies
$$||\sigma||^2 \geq \beta \cos^2\theta ||\nabla(\ln f)||^2.$$
\end{theorem}

\begin{theorem}\label{T:11.8}   Let $M=M_{\theta}\times_f F$ be a mixed geodesic warped product submanifold of an even-dimensional Kaehler manifold $\tilde M$ such that $M_{\theta}$ is a 
$\beta$-dimensional proper pointwise slant submanifold of $M$ and $F$ is a $\gamma$-dimensional totally real submanifold.  Then the squared norm of the second fundamental form $||\sigma||^2$ of $M$ satisfies
$$||\sigma||^2 \geq \beta \cot^2\theta ||\nabla(\ln f)||^2.$$
\end{theorem}

 \begin{remark} Theorem 8.1 of \cite{c03} was extended to warped product pointwise semi-slant submanifolds by A. Ali, S. Uddin and A. M. Othman in \cite{A}.
 
 \end{remark}

\end{document}